\documentclass[11pt]{article}

\usepackage[utf8]{inputenc}
\usepackage{amsmath,amsthm,amssymb,enumerate,framed,mdwlist}
\usepackage{color,colortab}
\usepackage[square,numbers,sort&compress]{natbib}
\usepackage[breaklinks=true,colorlinks,citecolor=blue,linkcolor=blue]{hyperref}

\newcommand{\Q}{\mathbb Q}
\newcommand{\N}{\mathbb{N}}
\newcommand{\Z}{\mathbb{Z}}
\newcommand{\R}{\mathbb{R}}
\newcommand{\conv}{\operatorname{conv}}

\newcommand{\cl}{\operatorname{cl}}

\newcommand{\ceil}[1]{\left\lceil#1\right\rceil}

\newcommand{\cC}{\mathcal{C}}
\newcommand{\A}{\mathcal{A}}

\newcommand{\cO}{\mathcal{O}}
\newcommand{\cT}{\mathcal{T}}
\newcommand{\cK}{\mathcal{K}}

\newcommand{\cH}{\mathcal{H}}

\renewcommand{\epsilon}{\varepsilon}

\newtheoremstyle{mythmstyle}
	{\topsep}
	{\topsep}
	{\itshape}
	{}
	{\scshape}
	{.}
	{3pt}
	{}
\theoremstyle{mythmstyle}

\newtheorem{nn}{}[section]
\newtheorem{lemma}[nn]{Lemma}
\newtheorem{theorem}[nn]{Theorem}
\newtheorem{cor}[nn]{Corollary}

\newtheorem{definition}[nn]{Definition}

\newtheorem{example}[nn]{Example}

\newtheorem{REMARK}[nn]{Remark}
\newenvironment{remark}{\begin{REMARK}}{\end{REMARK}}

\numberwithin{equation}{section}

\allowdisplaybreaks

\begin{document}

\title{Helly systems and certificates in optimization}

\author{Amitabh Basu\footnote{Department of Applied Mathematics and Statistics, The Johns Hopkins University. {\tt basu.amitabh@jhu.edu, tchen87@jhu.edu, hjiang32@jhu.edu}. The first and last authors gratefully acknowledge the support from AFOSR Grant FA95502010341 and NSF Grant CCF2006587.}\and Tongtong Chen\footnotemark[1]\and Michele Conforti\footnote{Dipartimento di Matematica ``Tullio Levi-Civita'', Universit\`a degli Studi Padova, Italy. {\tt conforti@math.unipd.it}}\and Hongyi Jiang\footnotemark[1]}
\maketitle

\begin{abstract} Inspired by branch-and-bound and cutting plane proofs in mixed-integer optimization and proof complexity, we develop a general approach via Hoffman's Helly systems. This helps to distill the main ideas behind optimality and infeasibility certificates in optimization. The first part of the paper formalizes the notion of a certificate and its size in this general setting. The second part of the paper establishes lower and upper bounds on the sizes of these certificates in various different settings. We show that some important techniques existing in the literature are purely combinatorial in nature and do not depend on any underlying geometric notions. 
\end{abstract}

\section{Certificates in optimization}

Let $U$ be an arbitrary set (we do not assume $U$ is a finite dimensional Euclidean space or has any geometric/vector space structure whatsoever). Let $\cK$ be a family of subsets of $U$ that contains the empty set and $U$, and is closed under arbitrary intersections. Following~\cite{hoffman1979binding}, we call the system $[U, \cK]$ a {\em Helly system}. 
The idea is to have an axiomatic definition that generalizes the family of convex sets in $\R^n$; see~\cite{van1993theory} for a comprehensive survey on axiomatic convexity with detailed references on the development of the subject. In particular, the notion of a {\em convex hull} goes over: for any $X \subseteq U$ (not necessarily in $\cK$), the convex hull of $X$, denoted by $\conv(X)$, is defined to be the set inclusion wise minimal set $K$ such that $X \subseteq K$, or equivalently, $\conv(X)$ is the intersection of all sets $K\in \cK$ such that $X \subseteq K$. An important case which is different from the family of all convex sets in $\R^n$ is the setting where $U=\Z^n$ and $\cK = \{K \subseteq \Z^n: K = C \cap \Z^n,\; \textrm{convex }C\}$\footnote{We will use $\N, \Z,$ and $\R$ to denote the set of natural numbers (starting at 1), the set of integers, and the set of real numbers respectively.}.

One can now define optimization problems over any Helly system. Typically, a subcollection $\cC\subseteq  \cK$ is chosen to define the structure of the constraints of the optimization problem (e.g., the subcollection of all halfspaces in $\R^n$ when the Helly system is the family of all convex sets in $\R^n$), and a collection $\cO$ of functions $f:U \to \R$ is selected such that all sublevels sets $S_{\leq\alpha} = \{x \in U: f(x) \leq \alpha\}$ and $S_{<\alpha} = \{x \in U: f(x) < \alpha\}$ are in $\cK$ (e.g., all convex functions on $\R^n$ when $U=\R^n$ and $\cK$ is the collection of all convex subsets); this defines the structure of allowable {\em objectives}. 

Let $K_1, \ldots, K_t \in \cC$ (the {\em constraints} of the problem instance) and let $f\in \cO$ (the {\em objective function} of the problem instance). We wish to solve \begin{equation}\label{eq:abs-opt}\inf \{f(x): x\in K_1 \cap \ldots \cap K_t\}.\end{equation}

Here are some geometric examples:
\begin{example}\label{ex:helly-sys}
\begin{enumerate}
\item $U= \R^n$, $\cK$ is the collection of all convex subsets of $\R^n$,  $\cC= \cK$ and $\cO$ is the collection of all convex functions. This is {\em convex optimization}. If $\cC$ is the collection of all closed halfspaces and $\cO$ is collection of all linear functions, we obtain {\em linear optimization}.
\item $U= \Z^n\times \R^d$, $\cK$ is the collection $\{C \cap U: C \subseteq \R^n\times \R^d \; \textrm{convex}\}$, $\cC= \cK$ and $\cO$ is the collection of all convex functions $\R^n \times \R^d$ restricted to $\Z^n \times \R^d$. This is {\em mixed-integer convex optimization}. One could again restrict $\cC = \{H \cap U: H \subseteq \R^n\times \R^d \; \textrm{closed halfspace}\}$ and $\cO$ to be linear functions (restricted to $U$) to get {\em mixed-integer linear optimization}.
\end{enumerate}
\end{example}

Several works have investigated the idea of optimizing over Helly systems beyond the classic convex optimization setting; see, for example,~\cite{de2018chance,amenta2015helly,de2016random,eisenbrand2003fast}. While the emphasis in prior work has been on algorithms and structural aspects of Helly systems, our focus in this paper is on the concept of a {\em certificate or proof of optimality} for~\eqref{eq:abs-opt}. In many settings, this boils down to establishing that $\gamma^*\in \R \cup \{+\infty\}$ is the optimal value by showing that the set $\{x\in U: f(x) < \gamma^*\}\cap K_1 \cap \ldots \cap K_t = \emptyset$. The concrete certificate that shows this set to be empty can be different for different settings of $[U,\cK]$, $\cC$ and $\cO$. One of the most widely used ones is {\em Farkas' lemma} in the setting of linear optimization (point 1. in Example~\ref{ex:helly-sys} above): $K_i = \{x\in \R^n: \langle a_i, x \rangle \geq b_i\}$, $i=1, \ldots, t$, $f(x) = \langle c, x\rangle$ with $a_1, \ldots, a_t, c \in \R^n$, $b_1, \ldots, b_t \in \R$. The Farkas certificate is a collection of nonnegative real numbers $u_1, \ldots, u_t$ such that $u_1a_1 + \ldots u_ta_t = c$ and $u_1b_1 + \ldots + u_tb_t = \gamma^*$, which provides a separating hyperplane between the polyhedron $K_1 \cap \ldots \cap K_t$ and the open halfspace $\{x\in U: \langle c, x \rangle < \gamma^*\}$. Similar certificates exist under certain conditions for nonlinear convex optimization which are called the Karush-Kuhn-Tucker (KKT) conditions. A slightly more general idea is the idea of proof of valid lower bounds. In the setting of linear optimization, a Farkas certificate of a valid lower bound $\bar\gamma$ is a collection of nonnegative real numbers $u_1, \ldots, u_t$ such that $u_1a_1 + \ldots u_ta_t = c$ and $u_1b_1 + \ldots + u_tb_t = \bar\gamma$, since this shows immediately that $\langle c, x \rangle \geq \bar\gamma$ for all points in $K_1 \cap \ldots \cap K_t$, or equivalently, that $K_1 \cap \ldots \cap K_t$ and the open halfspace $\{x\in U: \langle c, x \rangle < \bar\gamma\}$ are disjoint. We now formalize what we mean by a certificate of optimality or valid lower bound, and the notion of the size of such a certificate.


Our computing paradigm is the standard Turing machine model of computation. Following standard conventions, we restrict attention to optimization problems of the form~\eqref{eq:abs-opt} that have an appropriate encoding as a binary string and its length is the {\em encoding size} of the instance. We also consider the standard binary encodings of rational numbers. See~\cite{lovasz1986algorithmic,GroetschelLovaszSchrijver-Book88} for a careful and detailed discussion of these notions in the setting of nonlinear optimization, via the use of oracle Turing machines. For any binary string or rational number $s$, the length of the string or the number's binary encoding will be denoted by $|s|$. 

\begin{definition}\label{def:certificate} A class of optimization problems, defined by $[U,\cK]$, $\cC$ and $\cO$ is said to possess a {\em certificate of optimality} for every instance if there exists an algorithm (equivalently, Turing machine) $\A$ and a polynomial $p:\R \to \R$ with the following properties: 

\begin{enumerate}
\item For every instance $I$ and a lower bound $\bar\gamma \in \Q$ for $I$, i.e., $\bar\gamma$ is less than or equal to the optimal value of the problem, there exists a binary string $P$ such that when $(I, \bar\gamma, P)$ is fed as input to $\A$, the algorithm terminates in at most $p(|I| + |\bar\gamma| + |P|)$ steps reporting ``YES", and 
\item For any instance $I$ and a number $\bar\gamma \in \Q$ that is not a valid lower bound, there exists no binary string $P$ such that when $(I, \bar\gamma, P)$ is fed as input to $\A$, the algorithm terminates in at most $p(|I| + |\bar\gamma| + |P|)$ steps reporting ``YES".
\end{enumerate}

$P$ is called a {\em certificate for the lower bound $\bar\gamma$ for $I$}. The {\em size} of the certificate is defined to be $|P|$, and its {\em complexity} is defined to be the running time of the checker when executed on $(I, \bar\gamma, P)$.
\end{definition}

The restriction of polynomial running time for the checker $\A$ in the above definition ensures that the complexity of a certificate is at most a polynomial function of its size. This serves to avoid cases where the algorithm simply ignores the certificate and directly solves the problem. While this can still happen with the polynomial bound on running time, at least in that case we have a polynomial time algorithm for solving the optimization problems and therefore the execution of the algorithm itself can be taken as a ``certificate" of optimality. We also note that the definition above is very similar to certificates for defining the class NP in computational complexity, except that we do not impose a bound on the size of the certificate. This will be important for optimality certificates in the general Helly system setting, especially for contexts like mixed-integer optimization. We elaborate on this in Remark~\ref{rem:clarification}.

As we remarked above, an optimality certificate is often given in terms of infeasibility certificates. The above definition can be adapted to define certificates of infeasibility in a straightforward way: Case 1. will be for infeasible sets and Case 2. will be for feasible sets, or equivalently, one can consider the above definition with $\bar\gamma = +\infty$, where $+\infty$ is encoded/represented with a fixed binary string. Since the objective function plays no role for infeasibility, we will often say that we have a certificate of infeasibility for the system $K_1, \ldots, K_t$ (the constraints in~\eqref{eq:abs-opt}).

\begin{example}\label{ex:certificates} Let us revisit some well-known certificates.
\begin{enumerate}
\item We already discussed the {\em Farkas certificate} for linear optimization and the generalization to the KKT certificates in nonlinear optimization. In the linear case, the certificate itself is simply the nonnegative numbers $u_1, \ldots, u_t$ (by Caratheodory's theorem, at most $n$ of these multipliers need to be non zero) and the certificate's algorithm checks the equation $u_1a_1 + \ldots + u_ta_t= c$ and the inequality $u_1b_1 + \ldots + u_tb_t \geq \bar\gamma$. The  complexity is at most $n$ times the number of nonzero $u_i$'s. Note that in the case of linear optimization, any polynomial time algorithm for solving the optimization problem can be used as a checker with the empty string as the certificate. But the Farkas certificate is much more informative and the checker is much simpler than running a sophisticated optimization algorithm.
\item In mixed-integer linear optimization (point 2. in Example~\ref{ex:helly-sys}), if the problem is infeasible, a classical result says that there are at most $2^n(d+1)$ inequalities whose intersection already excludes all mixed-integer point. This idea has been used to develop a duality theory in~\cite{baes2015duality,basu2017optimality}. However, according to our definition, providing a list of such inequalities does not immediately translate into a certificate of size $2^n(d+1)$ (ignoring sizes of the numbers) since one has to design an algorithm that can check that indeed no mixed-integer point satisfies all these inequalities. No such algorithm is known that runs in time that is a polynomial function of $2^n(d+1)$. The algorithms with best known efficiency can solve the mixed-integer optimization problem itself and take time $2^{O(n\log(n+d))}$ which is quasi-polynomial in $2^n(d+1)$ and $n$.
\item Certificates for proving lower bounds in mixed-integer optimization based on cutting planes have a long history. Some of these cutting planes are generic, such as the {\em Chv\'atal-Gomory cutting planes}~\cite{sch}, and the certificate is a sequence of these cutting planes and the checker simply needs to verify the simple rounding computations for each cutting plane in the sequence. The checker's running time is at most $n+d$ times the size of the certificate. Some cutting planes are more specific in nature and use combinatorial information, e.g., the so-called {\em clique inequalities} for the {\em stable set problem}~\cite{conforti2014integer}. The certificate is simply a clique inequality and a list of vertices forming the corresponding clique and the checker verifies that the set of vertices indeed form a clique and the inequality has the correct form. The checker's running time is quadratic in the size of the certificate.
\end{enumerate}
\end{example}

\begin{remark}\label{rem:clarification} The size of a certificate of optimality/infeasibility need {\em not} be polynomial in the size of the instance. For example, in part of Example~\ref{ex:certificates} above, the Chv\'atal-Gomory cutting plane certificates are often provably exponential in the size of the input, see~\cite[Section 23.3]{sch} for a classic example. In our definition, only the checker's running time should be polynomial in the size of the certificate and the size of the instance.

It is also true that any (correct) algorithm for the optimization problem that finishes in finite time can be used to produce certificates as in Definition~\ref{def:certificate}. In particular, the execution of the algorithm itself can be taken as the certificate, i.e., an appropriate encoding of the sequence of elementary operations performed by the algorithm on the instance is the certificate and the checker simply verifies that this is indeed the execution transcript of the algorithm on this instance. The complexity of the certificate/checker will be linear in the size of such a certificate, and thus would satisfy Definition~\ref{def:certificate}. However, the main point is that Definition~\ref{def:certificate} allows for the possibility of certificates {\em whose size and complexity} are much smaller than the best known algorithmic complexity bounds (upper and lower) for the problem. Consider, for example, the case of linear optimization discussed in point 1. of Example~\ref{ex:certificates}. The Farkas certificate has size linear in the dimension and size of the instance and its complexity is quadratic in the dimension and linear in the size of the instance. However, no quadratic time algorithm is known for solving linear optimization.

This raises the question of whether one can prove a meta theorem which shows that the best complexity of an algorithm for any optimization problem is of the same order as the smallest size of a certificate (or complexity of the corresponding checker). We believe this would be a deep result in mathematical optimization, even in the specific case of (mixed-integer) linear optimization.
\end{remark}



\section{Embedded Helly systems and transfer of certificates}

For nonconvex settings like mixed-integer optimization, a certificate of infeasibility or valid lower bound often involves systematic reduction to a corresponding certificate for convex/linear optimization. These reduction techniques are known as branch-and-bound, or cutting plane, or more generally, branch-and-cut proof systems. We will now present them in our abstract Helly systems setting.

\begin{definition}\cite{hoffman1979binding} A Helly system $[U, \cK]$ is said to be {\em embedded} in a Helly system $[U',\cK']$ if $U \subseteq U'$ and $\cK = \{K' \cap U: K' \in \cK'\}$. $K' \in \cK'$ is said to be {\em relaxation} of $K \in \cK$ (in the system $[U', \cK']$) if $K = K' \cap U$. (Note that there may be $K'_1 \neq K'_2$ in $\cK'$ such that $K'_1 \cap U = K'_2 \cap U$. Thus, the same set in $\cK$ may have several different relaxations in the system $[U', \cK']$.) If $K'$ is a relaxation of $K$, then we also say that $K$ is a restriction of $K'$.

We denote this embedding by $[U, \cK] \subseteq [U', \cK']$.
\end{definition}

\begin{definition} Let $[U, \cK] \subseteq [U', \cK']$. A {\em valid disjunction} in $[U', \cK']$ for $[U, \cK]$ is a finite collection $K'_1, \ldots, K'_p \in \cK'$ such that $K_1 \cup \ldots \cup K_p = U$, where $K_i = K'_i \cap U \in \cK$, $i=1, \ldots, p$. The sets $K'_1, \ldots, K'_p$ are called the {\em terms} of the disjunction.
\end{definition}

\begin{definition} Let $[U, \cK]$ be a Helly system. A {\em halfspace} in $[U, \cK]$ is a set $K \in \cK$ such that $U \setminus K \in \cK$. We will use the notation $K^c$ to denote $U \setminus K$.
\end{definition}

Note that if $[U, \cK] \subseteq [U', \cK']$ and $H'$ is a halfspace in $[U', \cK']$, then $H = H' \cap U$ is a halfspace in $[U, \cK]$\footnote{It is possible that for an embedded system $[U, \cK] \subseteq [U', \cK']$, there may be $K \in \cK$ such that $K \in \cK'$ as well. The notation $K^c$ has an ambiguity because it could mean $U \setminus K$ or $U' \setminus K$. What we mean will be clear from context and we proceed with this slight overloading of notation.}. 

\begin{definition} Let $[U, \cK] \subseteq [U', \cK']$. Let $K \in \cK$. A {\em valid halfspace} in $[U', \cK']$ for $K$ is a halfspace $H'$ in $[U', \cK']$ such that $K \subseteq H$ where $H = H' \cap U$. Since $H$ is a halfspace in $\cK$, $H^c$ is also in $\cK$ and the condition can equivalently be stated as $H^c \cap K = \emptyset$.
\end{definition}


\begin{example}
\begin{enumerate}
\item (Chv\'atal-Gomory cutting planes) Let $U'=\R^n$, $\cK'$ be the collection of all convex sets in $\R^n$. Let $U = \Z^n$ and $\cK = \{H \cap U: H \in \cK'\}$. Thus, $[\Z^n, \cK]$ is embedded in $[\R^n, \cK']$. Consider any closed halfspace $H' \in \cK'$ that is rational (in the standard sense) and let $H = H' \cap U$. Let $H'' \in \cK$ be the inclusionwise minimal element in $\cK'$ such that $H \subseteq H''$. $H''$ is a valid halfspace for any $K$ such that $K$ has a relaxation $K' \subseteq H'$. $H''$ is called the Chv\'atal-Gomory cutting plane generated from $H'$. This idea can be generalized to any pair of embedded Helly systems $[U, \cK], [U', \cK']$ and starting with a halfspace in $\cK'$.


\item (Split disjunctions) Let $U = \Z^n$ and $\cK = \{C \cap U: C \textrm{ is a convex set in }\R^n\}.$ Let $H$ be a halfspace in $[U, \cK]$. Thus, $H$ and $U \setminus H$ are both in $\cK$. It can be verified that there exist halfspaces $H_1, H_2$ in $\R^n$ such that $H = H_1 \cap U$ and $U \setminus H = H_2 \cap U$. We let $H_1, H_2$ be any two closed halfspaces with this property. Then $H_1, H_2$ is a valid disjunction in $[\R^n, \cK']$ for $[U, \cK]$, where $\cK'$ is the collection of all convex sets in $\R^n$. If $H_1, H_2$ are rational and their union is not all of $\R^n$, then this is known as a split disjunction.
\end{enumerate}
\end{example}

The above examples illustrate how the idea of a valid halfspace generalizes the notion of a cutting plane in nonconvex optimization, and valid disjunctions generalize the corresponding notion in nonconvex optimization. 
We will now show that if we have an embedded Helly system $[U,\cK]$ in $[U', \cK']$, one can transfer infeasibility certificates from $[U', \cK']$ to $[U,\cK]$ using valid disjunctions and valid halfspaces.

\begin{definition}\label{def:BC-tree} A {\em branch-and-cut tree} for an embedded pair of systems $[U, \cK] \subseteq [U', \cK']$ is a rooted tree where each node has two labels and the following properties hold. 
\begin{enumerate}
\item For all nodes in the tree, the first label is a list $\mathcal{L}$ of sets in $\cK'$.
\item For internal nodes (nodes that are not leaves in the tree) the second label is either a valid disjunction in $[U', \cK']$ for $[U, \cK]$, or a valid halfspace for $K' \cap U$, where $K'$ is the intersection of all elements in $\mathcal{L}$.
\item For every leaf, the second label is a placeholder that can be empty or be populated with a certificate of infeasibility (see the next point).
\item Let $N$ be an internal node with labels $(\mathcal{L}, X)$, where $X$ is either a valid disjunction or a valid halfspace for $K' \cap U$, where $K'$ is the intersection of all elements in $\mathcal{L}$. If $X$ is a valid disjunction $K'_1, \ldots, K'_p$, then $N$ must have $p$ children with first labels $\mathcal{L} \cup \{K'_i\}$, $i=1, \ldots, p$. If the second label $X$ is a valid halfspace $H' \in \cK'$ for $K' \cap U$, then $N$ must have two children. The right child of $N$ must have first label $\mathcal{L} \cup \{H'\}$ and the left child, which is a leaf node,  has as second label a certificate of infeasibility for the system $M_1, \ldots, M_t, H^c$, where $H$ is the restriction of $H'$ and the $M_i$'s are the restrictions of the sets in $\mathcal{L}$. Such a leaf node is called a {\em cutting plane certifier}.
\end{enumerate}
If all second labels of all the internal nodes in the tree are valid disjunctions, then the tree is called a {\em (pure) branch-and-bound tree} and if the second labels are all valid halfspaces, then the tree is called a {\em (pure) cutting plane proof}. 


The {\em size} of the branch-and-cut tree is the total number of nodes in the tree plus the sum of all the sizes of all infeasibility certificates provided at the nodes with valid halfspaces. 
\end{definition}

\begin{remark}\label{rem:b-scheme-cp-paradigm} One can restrict the family of valid disjunctions and the family of valid halfspaces to be used. In the terminology of mixed-integer optimization, this would correspond to selecting a particular {\em branching scheme} for the valid disjunctions, and selecting a particular {\em cutting plane paradigm} for the valid halfspaces. The branch-and-cut trees are therefore relative to this choice. This will be important in our discussions below.
\end{remark}

\begin{remark}In point 3. of Definition~\ref{def:BC-tree}, we allow the possibility of the second label of a leaf node to be empty. This will be important for branch-and-cut certificates that are not infeasibility certificates; see Definition~\ref{def:hull-mem-valid-reverse}.\end{remark}

\begin{remark} We make some important comments about the ``certificate of infeasibility" for a node labeled with a valid halfspace. Typically, in certificates that involve cutting planes, the proof of validity of any cutting plane produced by the cutting plane paradigm is ignored. This is justified because the size of the proof of validity is usually linear or a low degree polynomial in the encoding size of the instance. Here, since we are dealing with general Helly systems and allowing for arbitrary cutting plane paradigms, we insist that the proof of validity of a cutting plane be explicitly included in the branch-and-cut tree by requiring a certificate of infeasibility for $K \cap H^c$ as defined above. 

Consider the concrete case of Chv\'atal-Gomory cutting planes. The proof of validity of such a cutting plane $\langle c, x \rangle \leq \delta$ for a polyhedron $P$ is the certificate of infeasibility of $P \cap \{\langle c, x \rangle > \delta\} \cap \Z^n$, by giving a certificate of infeasibility of $P \cap \{\langle c, x \rangle \geq \delta+1\}$ using a Farkas certificate. The same holds for disjunctive cuts like those derived from split disjunctions or lattice-free sets~\cite{conforti2014integer}.

Another example is the case of the clique inequalities for the stable set problem. The certificate of validity (or infeasibility of $K \cap H^c$) was discussed in point 3. of Example~\ref{ex:certificates}.
\end{remark}

We now formalize how to transfer, via branch-and-cut trees, certificates of infeasibility in $[U', \cK']$ to an embedded system $[U, \cK]$. The idea is that certificates of infeasibility in $[U', \cK']$ might be easier to obtain and a branch-and-cut tree can be used to provide a certificate in the more ``complicated" system $[U, \cK]$.

\begin{definition}\label{def:infeas-cert-transfer} Let $[U, \cK] \subseteq [U', \cK']$ be an embedded pair of systems. Let $K_1, \ldots, K_t \in \cK$ be such that $K_1 \cap \ldots \cap K_t = \emptyset$, and let $K'_1, \ldots, K'_t$  in $[U', \cK']$ be (arbitrary) relaxations of $K_1, \ldots, K_t$ respectively. A {\em branch-and-cut certificate of infeasibility} for the system $K_1, \ldots, K_t$ with respect to the relaxed system $K'_1, \ldots, K'_t$ is a branch-and-cut tree whose root node has $\mathcal{L} = \{K'_1, \ldots, K'_t\}$ as the first label and for every leaf, either it is a cutting plane certifier, or it has first label $\mathcal{L} = \{\widehat{K}_1, \ldots, \widehat{K}_m\}$ such that $\cap_{i=1}^m\widehat{K}_i = \emptyset$ and the second label is a certificate of infeasibility of the system $\widehat{K}_1, \ldots, \widehat{K}_m$ in $[U', \cK']$.

A checker for this certificate simply goes through the tree and verifies that the children at internal nodes are produced according to the valid disjunctions and cutting planes introduced, and verifies the cutting plane certifiers and the certificates of infeasibility at the leaves. The overall complexity of a branch-and-cut certificate is therefore the sum of the total nodes in the tree and the complexities of the cutting plane certifiers and the certificates of infeasibilities.\end{definition}
 
The following is a quantitative version of such a transfer of certificates, where the size of a certificate in $[U,\cK]$ is bounded in terms of the sizes of certificates in $[U',\cK']$.

\begin{theorem} Let $[U, \cK] \subseteq [U', \cK']$. Suppose there exists a function $f: \N \to \N$ be a function such that every certificate of infeasibility of any instance $I$ in $[U',\cK']$ is of size at most $f(|I|)$.\footnote{Note that we are implicitly assuming $f(n) \geq 1$ for all $n\in \N$. Thus, we are using the convention that a certificate in $[U',\cK']$ has size at least 1. Moreover, since we are dealing with certificates of infeasibility in this theorem, by an instance $I$ we mean a system of sets in $\cK'$ with empty intersection and the size $|I|$ is the encoding size of this system.} Suppose there is a branch-and-cut certificate of infeasibility for a system $K_1, \ldots, K_t$ with respect to the relaxed system $K_1', \ldots, K_t'$ in $[U', \cK']$ and all leaves have certificates of infeasibility in $[U', \cK']$, including the cutting plane certifiers. Then the branch-and-cut certificate has size at most $2L + \sum_{\textrm{ leaf }\ell}f(|\ell|) \leq 3\sum_{\textrm{ leaf }\ell}f(|\ell|)$, where $L$ is the number of leaves in the tree and $|\ell|$ denotes the size of the system given by the list of sets in the first label of the leaf $\ell$.
\end{theorem}

\begin{proof} This follows from the simple observation that a branch-and-cut tree has internal degree at least 3 (every internal node has at least 2 children) and so the number of nodes in the tree is at most 2 times the number of leaves. \end{proof}

\begin{remark} In Definition~\ref{def:infeas-cert-transfer}, we do allow the recursive possibility that the certificate of infeasibility of a left child of a node with a valid halfspace is a branch-and-cut certificate of infeasibility of $K \cap H^c$. This makes the notion of a ``pure cutting plane proof" a bit ambiguous. Nevertheless, we think this should not create any serious issues because the label ``pure cutting plane proof" should be reserved for situations where such a recursive definition is not really being used, but a direct certificate of infeasibility is provided at such nodes (e.g., in the system $[U',\cK']$, or via some combinatorial certificates like in the case of clique inequalities).
\end{remark}

We now consider three other useful notions of branch-and-cut certificates beyond infeasibility. These are inspired by notions defined and used in~\cite{dey2021lower}.

\begin{definition}\label{def:hull-mem-valid-reverse} Let $[U, \cK] \subseteq [U', \cK']$.
\begin{enumerate}
\item Let $K'_1, \ldots, K'_t \in \cK'$ be relaxations of $K_1, \ldots, K_t \in \cK$ respectively. For any relaxation $Q \in \cK'$ of $K_1 \cap \ldots \cap K_t$ such that $Q \subseteq K'_1 \cap \ldots \cap K'_t$, a branch-and-cut tree $\cT$ is said to {\em prove/certify $Q$ with respect to $K'_1, \ldots, K'_t$} if the root node has $\mathcal{L}:=\{K'_1, \ldots, K'_t\}$ as the first label and for every leaf of $\cT$, the intersection of the elements in the first label of the leaf is contained in $Q$.

Consider $P:= K_1 \cap \ldots \cap K_t$ as a subset of $U'$, then the {\em hull complexity} of $P$ with respect to $K'_1, \ldots, K'_t$ is the size of the branch-and-cut tree of smallest size that certifies $\conv(P)$ (defined in the Helly system $[U', \cK']$).
\item Let $K'_1, \ldots, K'_t \in \cK'$ be relaxations of $K_1, \ldots, K_t \in \cK$ respectively. Let $P:= K_1 \cap \ldots \cap K_t$. For any $x \in U' \setminus \conv(P)$, a branch-and-cut tree $\cT$ is said to {\em prove/certify $x$ with respect to $K'_1, \ldots, K'_t$} if the root node has $\mathcal{L}:=\{K'_1, \ldots, K'_t\}$ as the first label and $x$ is not contained in $\conv(Q)$ where $Q$ is the union, over all leaves, of the intersections of the elements in the first labels of the leaves. The {\em membership complexity} of $x$ is the size of the branch-and-cut tree of smallest size that certifies $x$. The {\em membership complexity} of $P:=K_1 \cap \ldots \cap K_t$ is the maximum membership complexity over all $x\not\in \conv(P)$.
\item Let $K'_1, \ldots, K'_t \in \cK'$ be relaxations of $K_1, \ldots, K_t \in \cK$ respectively. Let $P:= K_1 \cap \ldots \cap K_t$. For any valid halfspace $H'\in \cK'$ for $P$, a branch-and-cut tree $\cT$ is said to {\em prove/certify $H'$ with respect to $K'_1, \ldots, K'_t$} if the root node has $\mathcal{L}:=\{K'_1, \ldots, K'_t\}$ as the first label and for every leaf of $\cT$, the intersection of the elements in the first label of the leaf is contained in $H'$. The {\em validity complexity} of $P$ is the maximum validity complexity over all valid halfspaces in $\cK'$ for $P$.
\item Let $K'_1, \ldots, K'_t \in \cK'$ be relaxations of $K_1, \ldots, K_t \in \cK$ respectively. Let $P:= K_1 \cap \ldots \cap K_t$. For any valid halfspace $H'\in \cK'$ for $P$, the {\em reverse complexity} of $H'$ with respect to $K'_1, \ldots, K'_t$ is the size of the smallest branch-and-cut certificate of infeasibility for the system $K_1, \ldots, K_t, H^c$ with respect to $K'_1,\ldots, K'_t, H'^c$, where $H$ is the restriction of $H'$.
\end{enumerate}
\end{definition}

\begin{remark} As mentioned in Remark~\ref{rem:b-scheme-cp-paradigm}, one can restrict the family of valid disjunctions and valid halfspaces to be used in a branch-and-cut certificate. This obviously affects the sizes and therefore the corresponding notions of complexity laid out in Definitions~\ref{def:infeas-cert-transfer} and~\ref{def:hull-mem-valid-reverse}. In particular, for some restrictions on the valid disjunctions and halfspaces, no branch-and-cut tree certificate may exist for that particular complexity measure. In this case, we say the complexity measure is $+\infty$.
\end{remark}

\section{Helly numbers and bounds on the size of branch-and-cut infeasibility certificates}

Hoffman~\cite{hoffman1979binding} (see also~\cite{danzer1963helly}) defined the {\em Helly number} $h[U, \cK]$ of the Helly system $[U, \cK]$ to be the smallest natural number $h \in \N$ such that for any finite collection $K_1, \ldots, K_t \in \cK$, $K_1 \cap \ldots \cap K_t = \emptyset$ if and only if there exist $1 \leq i_1 \leq \ldots \leq i_h \leq t$ such that $K_{i_1} \cap \ldots \cap K_{i_h} = \emptyset$. If no such natural number exists, then $h[U, \cK] = \infty$. A collection $K_1, \ldots, K_t$ is called {\em critical} if the intersection is empty, but every strict subfamily has nonempty intersection. $h[U, \cK]$ is thus the size of the largest critical family ($\infty$ if there are critical families whose sizes increase without bound). 

It turns out that Helly numbers can provide lower bounds on the size of branch-and-cut certificates, in the case where the validity of cutting planes is proven using infeasibility certificates in the larger system (e.g., Chv\'atal-Gomory cutting planes or disjunctive cuts).

\begin{theorem}\label{thm:lower-BB-helly} Let $[U, \cK] \subseteq [U', \cK']$ and suppose the Helly number of $[U', \cK']$ is finite. Let $K_1, \ldots, K_t \in \cK$ be a critical family in $[U, \cK]$ and let $K'_i$ be any relaxation of $K_i$, $i=1, \ldots, t$ such that $K'_1 \cap \ldots \cap K'_t \neq \emptyset$, i.e., we have a nontrivial relaxation. Let $\cT$ be a branch-and-cut certificate of the infeasibility of the system $K_1, \ldots, K_t$, with respect to the relaxation $K'_1 \cap \ldots \cap K'_t$ in $[U', \cK']$ where the certificates on all the leaves are infeasibility certificates in $[U', \cK']$ (including cutting plane certifiers). Then, the certificate $\cT$ has size at least $\frac{t}{h[U',\cK']-1}$. In particular, one can find instances where branch-and-cut infeasibility certificates are of size at least $\frac{h[U,\cK]}{h[U',\cK']-1}$.
\end{theorem}

\begin{proof} 
We first observe that in the tree of the certificate $\cT$, for any node where the second label is a valid halfspace, one can view the two children as arising from a valid two-term disjunction, instead of a valid halfspace. More precisely, if $H' \in \cK'$ is the valid halfspace introduced at this node, then $H', (H')^c$ is a valid disjunction. Since all infeasibility certificates at the leaves of the tree are infeasibility certificates in $[U', \cK']$, in this way one can replace the entire branch-and-cut tree with a pure branch-and-bound tree where all second labels are valid disjunctions, with no change in the size of the tree. 

The second observation we make is that a pure branch-and-bound tree needs to keep track of only the valid disjunctions. In other words, we can completely forget about the first labels on each node and just keep the second labels. Such a tree can then be applied to any system $\widehat{K}_1, \ldots, \widehat{K}_m \in \cK'$ as the first label of the root node. If at every leaf of the new tree, the intersection of the elements in the first label is the empty set, then one obtains a new branch-and-cut certificate of infeasibility with respect to the new relaxed system $\widehat{K}_1, \ldots, \widehat{K}_m$.

Apply this tree to the trivial system $\mathcal{L} = \{U'\}$. Denote the intersections of the elements in first labels on all the leaves by $L'_1, \ldots, L'_s \in \cK'$.  Since all internal nodes in this new tree correspond to valid disjunctions, $\bigcup_{i=1}^s (L'_i \cap U) = U$.

Since $K_1, \ldots, K_t$ is a critical family for the Helly system $[U, \cK]$, for every $k \in \{1, \ldots t\}$ there exists $p^k \in \bigcap_{j \neq k} K_j = \bigcap_{j \neq k} K'_j \cap U$. Thus, each $p^k$ must end up in $L'_i$ for some $i\in \{1, \ldots, s\}$. For every $i=1,\ldots, s$, $L'_i \cap K'_1 \cap \ldots \cap K'_t = \emptyset$ because the corresponding leaf in $\mathcal{T}$ has additionally $K'_1, \ldots, K'_t$ in its first label, and $\cT$ is a branch-and-cut certificate of infeasibility with respect to $K'_1, \ldots, K'_t$. By definition of $h[U', \cK']$, there is a subfamily of $\{L_i',K_1',\ldots,K_t'\}$ with at most $h[U', \cK']$ size whose intersection is empty. If the subfamily does not include $L'_i$, then the corresponding subfamily of the $K'_i$ sets has empty intersection which contradicts the hypothesis of a nontrivial relaxation. 
Thus, we may assume for every leaf, $L'_i$ is included in the subfamily of size $h[U', \cK']$ that certifies infeasibility of the leaf. Therefore, at most $h[U', \cK']-1$ of the $K'_i$ sets are used to certify infeasibility at this leaf. This means that at most $h[U', \cK']-1$ points from the set $p^1, \ldots, p^t$ end up in this leaf $L_i$ because each $K'_i$ excludes at most $p^i$ from this set, and contains all other points $p^k$, $k\neq i$. Thus, we must have at least $\frac{t}{h[U',\cK']-1}$ leaves.
\end{proof}

\begin{remark} The above proof is a common generalization of some classical lower bounds on branch-and-cut and cutting plane proof sizes -- see, for example,~\cite{cook1990complexity,chvatal1989cutting,cook1987complexity} -- as well as recent lower bounds on branch-and-bound proof sizes~\cite{dadush2020complexity,dey2021lower}. In these papers, $U = \Z^n \subseteq \R^n = U'$, $\cK'$ is the set of all convex sets in $\R^n$ and $\cC'$ is the set of all open and closed halfspaces in $\R^n$. In other words, the setting is integer linear optimization. Since $h[\Z^n, \cK] = 2^n$~\cite{scarf1977observation,bell1977theorem,Doignon1973} and $h[\R^n, \cK'] = n+1$, we immediately obtain instances with lower bounds of $\frac{2^n}{n}$ on the size of branch-and-cut certificates. Different critical families were used in these references to obtain their bounds.

There has been a long line of work on cutting plane proof complexity; see~\cite{beame_et_al:LIPIcs:2018:8341,dash2002exponential,dash2005exponential,dash2010complexity,chvatal1989cutting,chvatal1984cutting,chvatal1980hard,cook2001matrix,bockmayr1999chvatal,eisenbrand2003bounds,rothvoss20130,bonet1997lower,razborov2017width,impagliazzo1994upper,buss1996cutting,cook1987complexity,goerdt1990cutting,goerdt1991cutting,clote1992cutting,pudlak1997lower,pudlak1999complexity,krajivcek1998discretely,grigoriev2002complexity} as a representative list. Complexity of branch-and-bound and branch-and-cut has received comparatively less attention, but with increased activity in recent years~\cite{dash2002exponential,dash2005exponential,basu-BB-CP,basu-BB-CP-II,fleming2021power,beame_et_al:LIPIcs:2018:8341,dey2020branch,cook1990complexity,dadush2020complexity,dey2021lower}.
\end{remark}

\begin{remark} \begin{enumerate} 
\item The above theorem makes no restrictions on the family of valid disjunctions used in the certificate. This means that the lower bounds on branch-and-bound found in~\cite{dadush2020complexity,dey2021lower} are not special to split disjunctions; they apply to {\em any} family of valid disjunctions that is used for branching. 

\item The only restriction on the valid halfspaces used in the branch-and-cut proof is that their validity is ultimately established using infeasibility certificates in the larger Helly system. In the context of mixed-integer optimization, this condition is satisfied by almost all known {\em general purpose} cutting planes since they are based on disjunctive arguments. Our definitions force the inclusion of a validity proof of a cutting plane into the branch-and-cut tree itself. This is often ignored in evaluating cutting plane proof sizes. This is justified when the validity certificates are short (linear or low degree polynomials of the instance size). However, in the general setting where we are not working with particular cutting plane paradigms, keeping track of the size of the validity certificates for cutting planes is appropriate for evaluating overall sizes of infeasibility/optimality certificates via branch-and-cut. Once we are in this setup, the above lower bound applies.
\end{enumerate}
\end{remark}

\begin{remark} Typically, lower bound proofs are established for Helly systems embedded in $[\R^n, \cK^\star]$, where $\cK^\star$ is the class of all convex sets in $\R^n$. 
But one can consider embedded systems like $[\Z^n, \cK] \subseteq [\Z^{n_1} \times \R^{n_2}, \cK']$, where $\cK = \{C \cap \Z^n: C \textrm{ convex}\}$, $\cK' = \{C \cap (\Z^{n_1}\times \R^{n_2}): C \textrm{ convex}\}$ and $n = n_1 + n_2$. Infeasibility certificates in $[\Z^{n_1} \times \R^{n_2}, \cK']$ can be obtained using methods like Lenstra's algorithm and its variants~\cite{Lenstra83} when $n_1$ is ``small". However, even with $n_1 = \Theta(\log n)$ (which is a regime where Lenstra's algorithm would still give super (quasi) polynomial size certificates), one obtains exponential lower bounds on the branch-and-cut certificates from Theorem~\ref{thm:lower-BB-helly}, since $h[\Z^{n_1} \times \R^{n_2}, \cK'] = 2^{n_1}(n_2 + 1)$~\cite{hoffman1979binding,AverkovWeismantel12}. Thus, $\frac{h[\Z^n, \cK]}{h[\Z^{n_1} \times \R^{n_2}, \cH]-1} = \frac{2^n}{n^c(n-c\log(n)+1)-1}$, which is exponential in $n$.
\end{remark}

\begin{remark} Since Theorem~\ref{thm:lower-BB-helly} applies to abstract Helly systems, it is clear that the explicit lower bounds in the literature based on these ideas appeal to geometric concepts only for getting bounds on the Helly numbers. The rest of the argument is purely combinatorial or set theoretic. This provides interesting insight into lower bound arguments in the literature for cutting plane/branch-and-bound/branch-and-cut proofs.
\end{remark}

\section{Relations between different measures of complexity}

We now explore the relationship between the measures of validity, reverse, hull and membership complexities from Definition~\ref{def:hull-mem-valid-reverse}. The inequalities below also trivially hold in the case of infinite complexity. We assume they are finite for simplicity.

To ease the burden of notation, in the rest of the paper instead of considering a system $K_1, \ldots, K_t \in \cK$ with corresponding relaxations $K'_1, \ldots, K'_t \in \cK'$, we will simply use $P := K_1 \cap \ldots \cap K_t$ to mean the system $K_1, \ldots, K_t$, and its relaxation $P':= K'_1 \cap \ldots \cap K'_t$ to mean the system $K'_1, \ldots, K'_t$. This does abuse notation, specially because the same sets $P$ (or $P'$) can be obtained as the intersection of two (or more) different systems of sets. However, since we will not be switching between different representations, this should not create any confusion below.

Moveover, for any node in a branch-and-cut tree, we will use the terminology {\em the set corresponding to the node} to mean the intersection of all the elements in the first label of the node.

\subsection{General Helly systems}


\begin{theorem} \label{thm:reverse-valid} Let $[U, \cK] \subseteq [U', \cK']$ and let $P' \in \cK'$ be a relaxation of $P \in \cK$. Let $H'\in \cK'$ be a valid halfspace for $P$. Suppose the family of disjunctions used in branch-and-cut tree certificates includes the simple valid disjunctions of $H, H^c$ for any halfspace $H$ in the larger Helly system. Then, 

\begin{center} reverse complexity of $H'$ $\leq$ validity complexity of $H'$ $\leq$ reverse complexity of $H'$ + 2.\end{center}
\end{theorem}

\begin{proof}
Let $\cT$ be a tree that proves $H'$. Then after applying this tree to $P' \cap (H')^c$ all the sets corresponding to leaves are empty (except possibly for the left children of nodes where a valid halfspace was used; here we have a certificate of infeasibility of such nodes). Thus, reverse complexity is less than or equal to validity complexity.

Suppose $\cT$ is a tree that proves that $P \cap H^c=\emptyset$ where $H = H' \cap U$. On $P'$, we first apply the valid disjunction $H'\cup (H')^c$. If we apply $\cT$ to the child correspond to $(H')^c$, all the sets corresponding to leaves are empty (except possibly for the left children of nodes where a valid halfspace was used; here we have a certificate of infeasibility of such nodes). Since $H'$ is by definition valid for the first child, we have a tree that proves validity of $H'$ using 2 more nodes than $\cT$. This gives the second inequality in the statement.
\end{proof}

\subsection{Helly systems in Euclidean space}

Much more can be said in Helly systems embedded in $[\R^n, \cK^\star]$, where $\cK^\star$ is the collection of all convex subsets of $\R^n$. The geometry of $\R^n$ has a lot more structure that can be utilized. We will use the notion of a facet of a polyhedron in $\R^n$ below. 

\begin{definition} Let $[U, \cK] \subseteq [\R^n, \cK^\star]$. Let $P \in \cK$ be such that $\conv(P)$ with respect to $\R^n$ is a polyhedron and let $P'$ be a relaxation of $P$. For any facet of $\conv(P)$, its {\em validity complexity} is the minimum of the validity complexities of all halfspaces that define that facet. The {\em facet complexity of $P$ with respect to $P'$} is the maximum validity complexity of all facets of $\conv(P)$.
\end{definition}

\begin{theorem}\label{thm:valid-facet}
Let $[U, \cK] \subseteq [\R^n, \cK^\star]$. Consider any subset $P \in \cK$ such that $P_I:= \conv(P)$ with respect to $\R^n$ is a full-dimensional polyhedron, and any relaxation $P'$ of $P$. Suppose the family of disjunctions used in branch-and-cut tree certificates includes the simple valid disjunctions of $H, H^c$ for any halfspace $H$ in the larger Helly system. Then, 
\begin{center} facet complexity of $P$ $\leq$ validity complexity of $P$ $\leq$ n(facet complexity of $P$ + 1).\end{center}
\end{theorem}

\begin{proof} 
The first inequality is clear. For the second inequality, by Caratheodory's theorem, any valid inequality $\langle c, x \rangle \leq \delta$ for $P_I$ is valid for a relaxation of $P_I$ obtained from $n$ facet defining halfspaces of $P_I$, say $H_i:= \{x \in \R^n: \langle a^i, x \rangle \leq b_i\}$, $i=1, \ldots, n$. 
We first do a disjunction $H_i, H_i^c$. Using the reverse complexity of $H_i$, we apply a tree on the second child to prove infeasibility. On the first child, we now apply the disjunction $H_2, H_2^c$ and again using the reverse complexity of $H_2$ we prove the infeasibility of the second child. Continuing in this manner, we finally end up with a leaf that whose first label is precisely the list of the facet defining halfspaces $H_1, H_2, \ldots, H_n$, and all other leaves are such that the corresponding sets are empty. This proves the validity of $\langle c, x \rangle \leq \delta$ with a tree of size at most the sum of the reverse complexities of $H_i$, $i=1, \ldots, n$ plus $n$ new nodes. Using the first inequality in Theorem~\ref{thm:reverse-valid} we get the second inequality in the statement.
\end{proof}

With a very similar proof, we also obtain the following variant of Theorem~\ref{thm:valid-facet} when we consider only pure cutting plane proofs.

\begin{theorem}\label{thm:valid-facet-2}
Let $[U, \cK] \subseteq [\R^n, \cK^\star]$. Consider any subset $P \in \cK$ such that $P_I:= \conv(P)$ with respect to $\R^n$ is a full-dimensional polyhedron, and any relaxation $P'$ of $P$. Suppose we disallow all disjunctions in the branch-and-cut tree, i.e., we consider only pure cutting plane proofs (the certificates of validity of the cutting planes may contain disjunctions). Then, 
\begin{center} facet complexity of $P$ $\leq$ validity complexity of $P$ $\leq$ n(facet complexity of $P$ - 1)+1.\end{center}
\end{theorem}

\begin{proof} 
The first inequality is clear. For the second inequality, by Caratheodory's theorem, any valid inequality $\langle c, x \rangle \leq \delta$ for $P_I$ is valid for a relaxation of $P_I$ obtained from $n$ facets of $P_I$, say $\langle a^i, x \rangle \leq b_i$, $i=1, \ldots, n$. Let $\mathcal{H}_i=\{H_{1,i},H_{2,i}\ldots, H_{N_i,i}\}$ be the sequence of all the cutting planes that appear in the cutting plane proof for $\langle a^i, x \rangle \leq b_i$, $i=1, \ldots, n$ (note that $N_i$ is one less than the size of the corresponding branch-and-cut tree with no disjunctions, i.e., $N_i$ is one less than the validity complexity of the facet defined by $\langle a^i, x \rangle \leq b_i$). If we combine all these cutting plane sequences $\mathcal{H}_1,\mathcal{H}_2, \ldots, \mathcal{H}_n$ we can derive a cutting plane proof for $\langle c, x \rangle \leq \delta$. Once we count the root node, we get the upper bound.\end{proof}

\begin{theorem}\label{thm:facet-hull} Let $[U, \cK] \subseteq [\R^n, \cK^\star]$. Consider any subset $P \in \cK$ such that $P_I:= \conv(P)$ with respect to $\R^n$ is a full-dimensional polyhedron, and any relaxation $P'$ of $P$. Suppose the family of disjunctions used in branch-and-cut tree certificates includes the simple valid disjunctions of $H, H^c$ for any halfspace $H$ in the larger Helly system. Then,

\begin{center} facet complexity of $P$ $\leq$ hull complexity of $P$ $\leq$ f(facet complexity of $P$ + 1)
\end{center} where $f$ is the number of facets of $P_I$.
\end{theorem}

\begin{proof} Since any tree that proves $P_I$ also proves the validity of any facet defining inequality, the first inequality in the statement follows. For the second inequality, we go through one facet at a time and do the following: If the facet is $H$, first employ the disjunction $H, H^c$. On the second child, use a tree that gives the reverse complexity for $H$ to prove infeasibility. Then continue the same thing with the remaining facets on the first child. Thus, we obtain a leaf whose first label is the list of all facets of $P_I$, and all other leaves have corresponding sets that are empty. Using the first inequality in Theorem~\ref{thm:reverse-valid}, the tree size is the sum of the validity complexities of all the facets + $f$. This gives the second inequality in the statement of this theorem.
\end{proof}

The upper and lower bounds in Theorem~\ref{thm:facet-hull} can be quite far apart if the instance has a large number of facets for the $P_I$. Nevertheless, the bounds cannot be improved -- see Section~\ref{sec:tight-examples}. If we restrict to pure cutting plane proofs, we get sharper bounds.



\begin{theorem}\label{thm:facet-hull-2}
Let $[U, \cK] \subseteq [\R^n, \cK^\star]$. Consider any subset $P \in \cK$ such that $P_I:= \conv(P)$ with respect to $\R^n$ is a full-dimensional polyhedron, and any relaxation $P'$ of $P$. Suppose we disallow all disjunctions in the branch-and-cut tree, i.e., we consider only pure cutting plane proofs (the certificates of validity of the cutting planes may contain disjunctions). Then, 
\begin{center} $\max\{$facet complexity of $P$, $f\}$ $\leq$ hull complexity of $P$ $\leq$ $f$(facet complexity of $P$ - 1) + 1\end{center} where $f$ is the number of nontrivial facets of $P_I$, i.e., those facets that are not valid for the for the relaxation $P'$.\end{theorem}
\begin{proof} Since every nontrivial facet must be derived at some point in any cutting plane proof of $P_I$, we get the first inequality.
Next we will show the second inequality in a similar manner as Theorem \ref{thm:valid-facet-2}. Let $\langle a^i, x \rangle \leq b_i$, $i=1, \ldots, f$ be all the nontrivial facets of $P_I$. Let $\mathcal{H}_i=\{H_{1,i},H_{2,i}\ldots, H_{N_i,i}\}$ be the sequence of all the cutting planes that appear in the cutting plane proof for $\langle a^i, x \rangle \leq b_i$, $i=1, \ldots, f$. If we combine all these cutting plane sequences $\mathcal{H}_1,\mathcal{H}_2, \ldots, \mathcal{H}_{f}$ we obtain a cutting plane proof for $P_I$. Once we count the root node, we can get the upper bound.
\end{proof}

\begin{remark}
Note that the logarithms of the lower bound $\max\{$facet complexity of $P$, $f\}$ and the upper bound $f$(facet complexity of $P$ - 1) + 1 are within a constant factor of each other, whereas the lower and upper bounds in Theorem \ref{thm:facet-hull} can be much farther apart, depending on the number of nontrivial facets. Section~\ref{sec:tight-examples} will establish that the bounds in Theorem \ref{thm:facet-hull} cannot be improved.
\end{remark}

\paragraph{Helly systems with integer points.} To establish a relationship between the membership complexity and the rest of the complexity measures, some more structure is required. The ideas are very much inspired by~\cite{dey2021lower}.

\begin{theorem}\label{thm:mem-facet} Let $[\Z^n, \cK] \subseteq [\R^n, \cK^\star]$ where $\cK = \{C \cap \Z^n: C \textrm{ convex}\}$. Let $P' \subseteq \R^n$ be a polyhedron and let $P = P' \cap \Z^n$ be the set of integer points in $P'$. Suppose further that $P_I := \conv(P)$ is a full-dimensional rational polytope. Suppose also that all split disjunctions are allowed as valid disjunctions in the branch-and-cut certificates. Then 

\begin{center} membership complexity of $P$ $\leq$ facet complexity of $P$ $\leq$ membership complexity of $P$ + 2.\end{center}
\end{theorem}

\begin{proof} Consider any $x^\star \not\in P_I$. There exists some facet of $P_I$ that separates $x^\star$ from $P_I$. The tree that proves the validity of this facet also proves $x^\star$. Thus, membership complexity is less than or equal to facet complexity.

Now consider any facet $F$ of $P_I$ given by the inequality $\langle a, x \rangle \leq b$, with $a \in \Z^n$ and $b\in \Z$. By Lemma 3 in~\cite{dey2021lower}, there exists $x^\star \in \R^n$ such that $\langle a, x^\star \rangle > b$ and for any point $\hat x \in P' \cap \{x: \langle a, x \rangle \geq b + 1\}$, $x^\star \in \conv(F \cup \{\hat x\})$. Let $s$ be the hull complexity of $\emptyset$ with respect to $P' \cap \{x: \langle a, x \rangle \geq b + 1\}$. By a similar argument as in the proof of Theorem~\ref{thm:reverse-valid}, $$s \geq \textrm{ (validity complexity of }\langle a, x \rangle \leq b) - 2.$$

If we consider any tree $\cT$ with size strictly less than $s$, then by definition of $s$ there will exist some point $\hat x \in P' \cap \{x: \langle a, x \rangle \geq b + 1\}$ that will be in the convex hull of the sets corresponding to leaves when this tree is applied to $P' \cap \{x: \langle a, x \rangle \geq b + 1\}$ and therefore also when applied to $P'$. In particular, $x^\star$ will not be eliminated. Hence, membership complexity of $x^\star$ is at least $s$. Therefore, 

\begin{center} membership complexity of $P$ $\geq$ membership complexity of $x^\star$ $\geq$ $s$ $\geq$   (validity complexity of $\langle a, x \rangle \leq b$) - 2.\end{center} 
Thus, 
\begin{center} (validity complexity of $\langle a, x \rangle \leq b$) $\leq$ membership complexity of $P$ + 2\end{center}
Taking the maximum of the left hand side over all facets, we obtain the second inequality in the statement of this theorem.
\end{proof}

With a similar proof, we also obtain the following variant of Theorem~\ref{thm:facet-hull} when we consider only pure cutting plane proofs, as long as the cutting plane paradigm contains all Chv\'atal-Gomory cuts.

\begin{theorem}\label{thm:mem-facet-2}
Let $[\Z^n, \cK] \subseteq [\R^n, \cK^\star]$ where $\cK = \{C \cap \Z^n: C \textrm{ convex}\}$. Let $P' \subseteq \R^n$ be a polyhedron and let $P = P' \cap \Z^n$ be the set of integer points in $P'$. Suppose further that $P_I := \conv(P)$ is a full-dimensional rational polytope. Suppose we disallow all disjunctions in the branch-and-cut tree, i.e., we consider only pure cutting plane proofs (the certificates of validity of the cutting planes may contain disjunctions), and the cutting plane paradigm used for the branch-and-cut certificates can generate all Chv\'atal-Gomory cuts. Then 

\begin{center} membership complexity of $P$ $\leq$ facet complexity of $P$ $\leq$ membership complexity of $P$ + 2.\end{center}
\end{theorem}

\begin{proof} Consider any $x^\star \not\in P_I$. There exists some facet of $P_I$ that separates $x^\star$ from $P_I$. The tree that proves the validity of this facet also proves $x^\star$. Thus, membership complexity is less than or equal to facet complexity.

Now consider any facet $F$ of $P_I$ given by the inequality $\langle a, x \rangle \leq b$, with $a \in \Z^n$ and $b\in \Z$. By Lemma 3 in~\cite{dey2021lower}, there exists $x^\star \in \R^n$ such that $\langle a, x^\star \rangle > b$ and for any point $\hat x \in P \cap \{x: \langle a, x \rangle \geq b + \frac{1}{2}\}$, $x^\star \in \conv(F \cup \{\hat x\})$. Let $s$ be the validity complexity of $\langle a, x \rangle \leq b + \frac{1}{2}$ (recall that this is with respect to a pure cutting plane proof now). If we consider any tree $\cT$ with size strictly less than $s$, then by definition of validity complexity, there will exist some point $\hat x \in P \cap \{x: \langle a, x \rangle \geq b + \frac{1}{2}\}$ that will be in the convex hull of the sets corresponding to leaves when this tree is applied to $P \cap \{x: \langle a, x \rangle \geq b + \frac{1}{2}\}$ and therefore also when applied to $P$. In particular, $x^\star$ will not be eliminated. Hence, membership complexity of $x^\star$ is at least $s$. Finally, by applying one Chv\'atal-Gomory strengthening of $\langle a, x \rangle \geq b + \frac{1}{2}$ we obtain the facet $F$. Note also that that proof of validity of this Chv\'atal-Gomory cutting plane has a Farkas certificate of size 1, since it is a trivial rounding of the right hand side. Thus, we have
\begin{center}
    facet complexity $\leq$ $s$+2 $\leq$ membership complexity +2, 
\end{center}
which finishes the proof. \end{proof}

\subsection{Tightness of the bounds} In the results above relating the different measures of complexity, the relationships in Theorems~\ref{thm:reverse-valid}, \ref{thm:mem-facet} and \ref{thm:mem-facet-2} are within an additive constant factor. However, the bounds in Theorems~\ref{thm:valid-facet}, \ref{thm:valid-facet-2}, \ref{thm:facet-hull}, and \ref{thm:facet-hull-2} are not as close to each other. We now show that these bounds cannot be improved further, except possibly up to constant factors. In Section~\ref{sec:split-cover} below, we develop some tools for this analysis. The examples establishing the tightness of the bounds appear in Section~\ref{sec:tight-examples}.

Our examples are in the embedded Helly systems $[\Z^n, \cK] \subseteq [\R^n, \cK^\star]$ where $\cK^\star$ is the collection of all convex sets in $\R^n$ and $\cK = \{C \cap \Z^n: C \textrm{ convex}\}$, i.e., we consider examples involving integer points in convex sets.

\subsubsection{Split cover number and the lower bounds on certificate size}\label{sec:split-cover}


\begin{definition}
We define a {\em generalized split (g-split)} in $\R^n$ as a set of the form $S:=\{x\in\R^n:\beta< \alpha x< \beta+1\}$ for some $\alpha\in \Z^n\backslash\{{\bf 0}\}$, and $\beta\in \R$. $S$ is simply called a {\em split set} if $\beta\in \Z$. Note that generalized split sets are open sets in $\R^n$.

A {\em split disjunction} corresponding to a split set is the valid disjunction $\{x\in\R^n: \alpha x\leq \beta\} \cup \{x\in\R^n: \alpha x\geq \beta + 1\}$ for the embedded Helly systems $[\Z^n, \cK] \subseteq [\R^n, \cK^\star]$, with $\cK, \cK^\star$ as defined above.
\end{definition}

\begin{definition}
Let $X$ be a set in $\R^n$ such that $X\cap \Z^n=\emptyset$. We say {\em $X$ is covered by a set $\mathcal{S}$} of split sets in $\R^n$ if $X\subseteq \bigcup_{S\in\mathcal{S}}S$. The {\em split cover number of $X$} is defined as the minimum size amongst all sets of split sets that can cover $P$.
\end{definition}


\begin{lemma}\label{lem:covering-hull}
Consider any set $P\subseteq \R^n$ such that the split covering number of $P\setminus \conv(P\cap \Z^n)$ is $\ell$. Then, the hull complexity of $\conv(P\cap \Z^n)$ with respect to $P$ is at least $2\ell+1$ when using branch-and-cut trees where all valid disjunctions and cutting planes are based on a family of split disjunctions.
\end{lemma}

\begin{proof}
If a branch-and-cut tree $\mathcal{T}$ certifies $\conv(P\cap \Z^n)$, and $\mathcal{S}$ is the set of all the split sets corresponding to the split disjunctions that are used to branch or derive a cutting plane in $\mathcal{T}$, then we must have $P\setminus \conv(P\cap \Z^n)\subseteq \bigcup_{S\in\mathcal{S}}S$. At a node $N$ of $\mathcal{T}$ with corresponding set $K$, if we apply branching, then it produces two nodes in $\mathcal{T}$. If we apply a cutting plane $H$ derived from a split disjunction $D$ on $N$, then the right child has $K\cap H$ as the corresponding set and the left child has $K\cap H^c$ as the corresponding set, which is a subset of the split set $S$ corresponding to $D$. The complexity for infeasibility certificate of $K \cap H^c$ is at least $2$. Including the root node, we are done. 
\end{proof}

\subsubsection{Tight examples}\label{sec:tight-examples}


\paragraph{Facet versus hull complexity.} We now construct two examples establishing that the bound in Theorems~\ref{thm:facet-hull} cannot be improved.

In Example~\ref{ex:facet=hull} below, we have a family of examples which have exponentially many facets all of which have $O(1)$ validity complexity. Thus, in Theorems~\ref{thm:facet-hull} the gap between the upper and lower bounds  is exponential. Nevertheless, the hull complexity is still $O(1)$, showing that the lower bound can be tight, even when the upper bound is exponentially larger. 

\begin{example}\label{ex:facet=hull}
For any natural number $n\geq 2$, we consider a polytope $P\in \R^2$ with $2^n$ integral vertices which are denoted by $(x_1^{(1)}, x_2^{(1)}), (x_1^{(2)}, x_2^{(2)}), (x_1^{(3)}, x_2^{(3)})\ldots,(x_1^{(2^n)}, x_2^{(2^n)})\in \Z^2$ ordered in clockwise direction, and therefore, $P=P_I$. Let $a_1^{(1)}x_1^{(1)}+ a_2^{(1)}x_2^{(1)}\leq b_1, a_1^{(2)}x_1^{(2)}+ a_2^{(2)}x_2^{(2)}\leq b_2,\ldots, a_1^{(2^n)}x_1^{(2^n)}+ a_2^{(2^n)}x_2^{(2^n)}\leq b_{2^n}$ be the $2^n$ facet-defining halfspaces of $P$, where $a_j^{(i)}\in \Z$ and $b_i\in \Z$ for $i=1,2,\ldots,2^n$ and $j=1,2$. Let $v'_j:=(v_1^{(j)}, v_2^{(j)})$ be such that $a_1^{(j)}v_1^{(j)}+ a_2^{(j)}v_2^{(j)}>b_j$ and define $v_j = (v_1^{(j)}, v_2^{(j)}, \frac{1}{2}, \frac{1}{2}, \ldots, \frac{1}{2} )\in \R^n$ for $j=1, 2, 3, \ldots, 2^n$. 
Let $P^{(n)}:=\conv((P\times [0,1]^{n-2})\cup \{v_1, v_2, v_3, \ldots, v_{2^n}\})$. 

Then, $\conv(P^{(n)}\cap \Z^n) = P\times [0,1]^{n-2}$. Moreover, the hull complexity and the validity complexity of any facet of $P\times [0,1]^{n-2}$ with respect to $P^{(n)}$ is at most $3$ for branch-and-cut proofs where the valid disjunctions are based on the family of variable disjunctions, i.e., disjunctions of the type $\{x\in \R^n: x_i \leq \beta\} \cup \{x\in \R^n: x_i \geq \beta + 1\}$ for some $\beta\in \Z$.
\end{example}
\begin{proof}
Since $P^{(n)}\subseteq \{x\in \R^n: 0\leq x_3\leq 1\}$, so $P^{(n)}\cap \Z^n\subseteq \{x\in \R^n: x_3= 1\}\cup \{x\in \R^n: x_3= 0\}$. Also, $\{v_1, v_2, v_3, \ldots, v_{2^n}\}\subseteq \{x\in \R^n: x_3= \frac{1}{2}\}$. Thus $\conv(P^{(n)}\cap \Z^n) = P\times [0,1]^{n-2}$. By assumption,  $\{v_1, v_2, v_3, \ldots, v_{2^n}\}\not\subseteq  P\times [0,1]^{n-2}$, so $P^{(n)}\supsetneq P\times [0,1]^{n-2}$. If we apply the valid disjunction $\{x\in \R^n: x_3 \leq 0\}\cup\{x\in \R^n: x_3\geq 0\} $, the convex hull of the sets corresponding to leaves is $P\times [0,1]^{n-2} = \conv(P^{(n)}\cap \Z^n)$, which finishes the proof. 
\end{proof}

In Theorem~\ref{thm:large-hull-complexity} below, we give a family of examples in $\R^n$ where the number of facets is $n$, the validity complexity of each facet is $O(1)$ (they are simple Chv\'atal-Gomory cutting planes), but the hull complexity is $\Omega(n)$. This shows that the upper bounds in Theorems~\ref{thm:facet-hull} and~\ref{thm:facet-hull-2} can be tight, even if the gap between the lower and upper bounds is linear.

\begin{lemma}\label{lem::cube-cover}
Let $d\in \N$. For every $k\in \N$, define $P_k:=[0, 2k+1]^d$. Let $S_i$ be a g-split set for $i=1,\ldots, 2k$. Then $P_k\not\subseteq \bigcup_{i=1}^{2k}\cl(S_i)$, where $\cl(X)$ means the closure of $X$.  
\end{lemma}

\begin{proof}
We will use induction on the dimension $d$ to prove the claim. When $d=1$, this is clear since any g-split is an interval of length at most 1. Assume the claim holds when $d=n-1$. Next we will show that this is true when $d=n$. Let $I\subseteq\{1,\ldots,2k\}$ be the set of indices such that $S_i=\{x:m< x_1 < m+1\}$ for some $m\in \R$ if $i\in I$. Since $|I|\leq 2k$, we can find some $m_0\in [0, 2k+1]$ such that $\{x:x_1= m_0\}\not\subseteq \cl(S_i)$ for each $i\in I$. For $i\in\{1,\ldots,2k\}\backslash I$, $S_i\cap \{x:x_1=m_0\}$ is a g-split set in the $n-1$ dimensional affine subspace defined by $\{x:x_1=m_0\}$. Then by the induction hypothesis and the definition of $m_0$, we have $P_k\cap \{x:x_1=m_0\}\not\subseteq \bigcup_{i=1}^{2k}\cl(S_i)$, which finishes the proof.
\end{proof}

\begin{lemma}\label{lem:intersect-lower-bound}
Let $n\in \N$. Consider the subset in $\R^n$ defined by $P':=[0,2n+1]^{n-1}\times (2n+1,2n+1+\epsilon)$ for $\epsilon>0$. Let $S_i$, $i=1, \ldots, n$ be split sets that all have nonempty intersection with $\{x\in\R^n:x_n=2n+1\}$. Then $P'\not\subseteq \left(\bigcup_{i=1}^{2n}S_i\right)$.
\end{lemma}
\begin{proof}
Let $\cl({S}_i)$ be the closure of $S_i$ for $i=1,\ldots, 2n$ and observe that each $\cl({S}_i)\cap \{x:x_n=2n+1\}$ is the closure of a split set in the $n-1$ dimensional affine subspace defined by $\{x:x_n=2n+1\}$.  Then by Lemma \ref{lem::cube-cover}, $[0,2n+1]^{n-1}\times\{2n+1\}\not\subseteq \left(\bigcup_{i=1}^{2n}\cl({S}_i)\right)$. Since $\left(\bigcup_{i=1}^{2n}\cl({S}_i)\right)$ is a closed set, so there exists $v'\in [0,2n+1]^{n-1}\times\{2n+1\}$ such that $\{x\in\R^n:\|x-v'\|_2\leq \epsilon'\}\cap\left(\bigcup_{i=1}^{2n}\cl({S}_i)\right)=\emptyset$ for some $\epsilon'>0$. Since $v'\in[0,2n+1]^{n-1}\times\{2n+1\}$, so we have $\{x\in \R^n:\|x-v'\|_2\leq \epsilon'\}\cap P'\neq \emptyset$, and $\{x\in\R^n:\|x-v'\|_2\leq \epsilon'\}\cap P'\cap \left(\bigcup_{i=1}^{2n}S_i\right)=\emptyset$, which finishes the proof.
\end{proof}

\begin{theorem}\label{thm:covering-number-1}
Let $n\in \N$ and define $P:=[-\frac{1}{2}, \frac{3}{2}+2n]^n$. Then the split covering number of $P \setminus \conv(P\cap \Z^n)$ is $2n$. 
\end{theorem}

\begin{proof}
It is clear that $P_I:= \conv(P\cap \Z^n)=[0,2n+1]^n$. 
Let $P_{i1}:=\{x:2n+1<x_i<2n+\frac{3}{2},0<x_j<2n+1, \mbox{for }j\neq i,j=1,\ldots, n\}$, $P_{i2}:=\{x:-\frac{1}{2}<x_i<0,0<x_j<2n+1, \mbox{for }j\neq i,j=1,\ldots, n\}$. Also define the split sets $C_{i1}=\{x: 2n+1<x_i<2n+2\}$, and $C_{i2}=\{x:-1<x_i<0\}$ for $i=1,\ldots, n$. Since $P\backslash P_I = \bigcup_{i=1}^n\bigcup_{j=1}^2P_{ij} \subseteq \bigcup_{i=1}^n\bigcup_{j=1}^2C_{ij}$, so the split covering number of $P\backslash P_I$ is at most $2n$. Next we will show that it must be at least $2n$. 

Consider the split sets that cover $P_{i1}$ for some $i \in \{1, \ldots, n\}$. By Lemma \ref{lem:intersect-lower-bound}, if $P_{i1}$ is covered by only split sets that have nontrivial intersection with $\{x\in\R^n:x_i=2n+1\}$, then we need more than $2n$ splits. Thus, we must use at least one split set that has empty intersection with $\{x\in\R^n:x_i=2n+1\}$ in the cover for $P_{i1}$. The same goes for $P_{i2}$: we must use at least one split set that has empty intersection with $\{x\in\R^n:x_i=0\}$. Since the choice of $i \in \{1, \ldots, n\}$ was arbitrary, this shows we need at least $2n$ split sets to cover $\bigcup_{i=1}^n\bigcup_{j=1}^2P_{ij} = P\backslash P_I$.
%
\end{proof}





\begin{theorem}\label{thm:large-hull-complexity}
Let $P$ be $[-\frac{1}{2}, \frac{3}{2}+2n]^n$. Then its facet complexity is $3$ and its hull complexity is $4n+1$ in terms of any branch-and-cut certificate based on any subfamily of split disjunctions that includes all variable disjunctions. 
\end{theorem}

\begin{proof}
It is clear that $P_I:= \conv(P\cap \Z^n)=[0, 2n+1]^n$. Each facet-defining halfspace of $P_I$ is a Chv\'atal-Gomory cut. Thus its facet complexity is $3$. 

Since $P_I$ can be derived by $2n$ Chv\'atal-Gomory cuts defined by the facet-defining halfspaces, so the hull complexity is at most $4n+1$. By Theorem \ref{thm:covering-number-1}, we need at least $2n$ split sets to cover $P\backslash P_I$, so the lower bound of hull complexity is $4n+1$ due to Lemma \ref{lem:covering-hull}, which finishes the proof. 
\end{proof}


\paragraph{Facet versus validity complexity.} We next give an example in Theorem~\ref{thm:facet<valid} below showing that the upper bounds in Theorems~\ref{thm:valid-facet} and~\ref{thm:valid-facet-2} cannot be improved any further. In particular, the example has facet complexity $O(1)$ and a certain valid inequality is shown to have validity complexity of at least $\Omega(n)$.

\begin{lemma}\label{lem:simplex-cover}
Let $P(n,b)=\{x\in \R^n: x_i\geq \epsilon, \mbox{ for }i=1,2,3,\ldots, n,~    \sum_{i=1}^nx_i <\;n\epsilon+ b\}$ for some $0\leq\epsilon\leq 1$, and $b>0$. Let $\mathcal{S}$ be a set of $\ceil{b}-1$ g-split sets. Then $P(n,b)\not\subseteq\bigcup_{S\in\mathcal{S}}S$. 
\end{lemma}


\begin{proof}
When $b\leq 1$, it is trivial, so we assume $b> 1$. We will prove the claim by induction on the dimension $n$. 

For $n=1$, $P(1,b)=\{x\in \R:\epsilon\leq x<\epsilon+ b\}$. Thus the claim is true in this case since any g-split is an interval of length at most 1. Assume it is true for $n=d-1$, $d\geq 2$. Then when $n=d$, let $S_i$ be a g-split in $\R^n$ for $i=1,\ldots,\ceil{b}-1$. Suppose $I\subseteq \{1,\ldots, \ceil{b}-1\}$ is a set of all the indices such that $S_i=\{x:b'< x_n< b'+1\}$ for $i\in I$ and $b'\in \R$. Since $|I|\leq \ceil{b}-1$, we have $\cl(P(n,b))\backslash (\bigcup_{i\in I}S_i)\neq \emptyset$ and $\cl(P(n,b))\backslash (\bigcup_{i\in I}S_i)$ is a compact set. Thus, $b'' := \min\{x_n : x \in \cl(P(n,b))\backslash (\bigcup_{i\in I}S_i)\}$ is a well-defined real number. Also, since $|I|\leq \ceil{b}-1$, we have $\epsilon\leq b''\leq \epsilon+\ceil{b}-1<n\epsilon+b$. Thus, $P(n,b)\cap \{x\in\R^n:x_n=b''\}\neq \emptyset$. 

Note that we must have $|I|\geq \ceil{b''-\epsilon}$. Let $\mathcal{L}$ be the $n-1$ dimensional affine subspace defined by $\{x\in\R^n:x_n=b''\}$. Then $P(n,b)\cap \{x\in\R^n:x_n=b''\}=\{x\in\R^n:x_n=b'',x_i\geq \epsilon, \mbox{ for }i=1,\ldots, n-1, \sum_{i=1}^{n-1}x_i< n\epsilon + b-b''\}$, which is $P(n-1, b-b''+\epsilon)$ in $\mathcal{L}$. Since $|\{1,\ldots, \ceil{b}-1\}\backslash I|\leq \ceil{b}-1 - \ceil{b''-\epsilon}\leq \ceil{b-b''+\epsilon}-1$, by the induction hypothesis,  we have $P(n,b)\cap \{x\in\R^n:x_n=b''\}\not\subseteq\bigcup_{i\in\{1,\ldots, \ceil{b}-1\}\backslash I}S_i$. By the definition of $b''$, we have that  $P(n,b)\cap \{x\in\R^n:x_n=b''\}\not\subseteq\bigcup_{i\in\{1,\ldots, \ceil{b}-1\}}S_i$, which finishes the proof. 
\end{proof}

\begin{cor}\label{cor:simplex-split-cover}
Suppose $P'=\{x\in \R^n: x_i\geq \frac{1}{2n}, \mbox{ for }i=1,2,3,\ldots, n,~    \sum_{i=1}^nx_i<n\}$. The split covering number of $P'$ is at least n. 
\end{cor}
\begin{proof}
Follows from Lemma \ref{lem:simplex-cover} with $\epsilon=\frac{1}{2n}$ and $b=n-\frac{1}{2}$.
\end{proof}

\begin{theorem}\label{thm:facet<valid}
Suppose $P=[\frac{1}{2n}, 2+n-\frac{1}{2n}]^n$. Then $P_I=[1,n+1]^n$, all facets are simple Chv\'atal-Gomory cuts with validity complexity $O(1)$, and the validity complexity of $\sum_{i=1}^nx_i\geq n$ is $2n+1$ with respect to branch-and-cut trees based on any subfamily of split disjunctions that includes all the variable disjunctions. 
\end{theorem}

\begin{proof}
It is clear that $P_I=[1,n+1]^n$, and all facets are obtained by rounding the right hand sides of the bound inequalities and thus the facet complexity is $O(1)$. Let $P'=\{x\in \R^n: x_i\geq \frac{1}{2n}, \mbox{ for }i=1,2,3,\ldots, n,~ \sum_{i=1}^nx_i<n\}$. By Theorem~\ref{thm:reverse-valid}, the validity complexity of $\sum_{i=1}^nx_i\geq n$ is at least its reverse complexity, which is the hull complexity of $\emptyset$ with respect to $P'$ (the analysis of the first inequality in Theorem~\ref{thm:reverse-valid} does not require all halfspace disjunctions). By Corollary \ref{cor:simplex-split-cover}, we need at least $n$ split sets $S_j$ for $j=1,\ldots, n$ such that $P'\subseteq \bigcup_{j=1}^{n}S_i$. Thus the hull complexity of $\emptyset$ with respect to $P'$ is at least $2n+1$ by Lemma \ref{lem:covering-hull}, and so the validity complexity of $\sum_{i=1}^nx_i\geq n$ is at least $2n+1$. This lower bound can be attained by applying valid disjunctions $\{x\in \R^n:x_i\leq 0\} \cup\{x \in \R^n: x_i \geq 1\}$ for $i=1,\ldots,n$, which finishes the proof. 
\end{proof}

\begin{remark}
The above analysis shows that the split cut proof length of the valid inequality $\sum_{i=1}^nx_i\geq n$ is exactly $n+1$. 
\end{remark}



\bibliographystyle{plain}
\bibliography{../full-bib}

\end{document}